\documentclass[11pt]{amsart}
\usepackage{amsmath}
\usepackage{amsthm}
\usepackage{hyperref}
\usepackage[margin=1.0in]{geometry}
\usepackage{mathrsfs}
\numberwithin{equation}{section}
\newtheorem{prop}{Proposition}
\newtheorem{lemma}[prop]{Lemma}
\newtheorem{thm}[prop]{Theorem}
\newtheorem{cor}[prop]{Corollary}

\numberwithin{prop}{section}
\theoremstyle{definition}

\newtheorem{rmk}[prop]{Remark}
\newcommand{\Fp}{F^{+}}                                                 %
\newcommand{\Wp}{W^{+}}                                                 
\newcommand{\hg}{\hat{g}}
\newcommand{\ten}{\otimes}                                               %
\newcommand{\brs}[1]{\left| #1 \right|}
\newcommand{\brk}[1]{\left[ #1 \right]}
\newcommand{\prs}[1]{\left( #1 \right)}
\newcommand{\sqg}[1]{\left\{ #1 \right\}}
\newcommand{\ip}[1]{\left\langle #1 \right\rangle}
\newtheorem{innercustomthm}{\textbf{Theorem}}
\newenvironment{customthm}[1]
{\renewcommand\theinnercustomthm{#1}\innercustomthm}
  {\endinnercustomthm}

\newcommand{\del}{\partial}

\newcommand{\gG}{\Gamma}
\renewcommand{\gg}{\gamma}

\newcommand{\gD}{\Delta}
\newcommand{\gd}{\delta}

\newcommand{\gk}{\kappa}

\newcommand{\gw}{\omega}
\newcommand{\ga}{\alpha}
\newcommand{\gb}{\beta}
\renewcommand{\ge}{\epsilon}
\newcommand{\N}{\nabla}

\newcommand{\lap}{\Delta}

\newcommand{\dVg}{\dV_g}
\newcommand{\Eul}{\chi}

\DeclareMathOperator{\Ric}{Ric}

\DeclareMathOperator{\R}{R}

\DeclareMathOperator{\tr}{tr_{L^2}}

\DeclareMathOperator{\Vol}{Vol}

\DeclareMathOperator{\SU}{SU}
\DeclareMathOperator{\SO}{SO}
\DeclareMathOperator{\cZ}{Z}

\DeclareMathOperator{\cd}{d}
\DeclareMathOperator{\End}{End}

\DeclareMathOperator{\dV}{dV}
\DeclareMathOperator{\ds}{ds}
\DeclareMathOperator{\rank}{rank}
\DeclareMathOperator{\cY}{Y}
\DeclareMathOperator{\trace}{tr}

\begin{document}
\title[Index-energy estimates for Yang--Mills connections and Einstein metrics]{Index-energy estimates for Yang--Mills connections and Einstein metrics}
\author{Matthew J. Gursky}
\address{Department of Mathematics \\
          255 Hurley Bldg\\
	  University of Notre Dame\\
         Notre Dame, IN  46556}
\email{\href{mailto:mgursky@nd.edu}{mgursky@nd.edu}}
\author{Casey Lynn Kelleher}
\address{Department of Mathematics
         Princeton University\\
         Princeton, New Jersey, 08540}
\email{\href{mailto:ckelleher@princeton.edu}{ckelleher@princeton.edu}}
\author{Jeffrey Streets}
\address{Department of Mathematics
	    Rowland Hall,
         University of California, Irvine\\
         Irvine, 92697}
\email{\href{mailto:jstreets@uci.edu}{jstreets@uci.edu}}
\thanks{M.J. Gursky is supported by NSF grant DMS-1811034.  C.L. Kelleher is supported by a National Science Foundation Postdoctoral Research Fellowship.  J. Streets is supported by NSF grant DMS-1454854}

\begin{abstract} We prove a conformally invariant estimate for the index of Schr\"odinger operators acting on vector bundles over four-manifolds, related to the classical Cwikel--Lieb--Rozenblum estimate.  Applied to Yang--Mills connections we obtain a bound for the index in terms of its energy which is conformally invariant, and captures the sharp growth rate.  Furthermore we derive an index estimate for Einstein metrics in terms of the topology and the Einstein--Hilbert energy.  Lastly we derive conformally invariant estimates for the Betti numbers of an oriented four-manifold with positive scalar curvature.
\end{abstract}
\date{January 14th, 2019}
\maketitle

\section{Introduction}

The classical Cwikel--Lieb--Rozenblum (CLR) estimate \cite{Cwikel,Lieb, Rozenblum}, related to the famous asymptotic formula of Weyl \cite{Weyl} on the growth of eigenvalues, bounds the index of a Schr\"odinger operator $L = - \gD + V$ on a bounded domain in $\mathbb R^n$ in terms of the $L^{\frac{n}{2}}$ norm of the negative part of $V$.  This central result has applications to mathematical physics, where it is referred to as an estimate of the number of bound states for the linear Schr\"odinger operator.  From the point of view of both geometry and mathematical physics, it is important to find similar index/bound state estimates for nonlinear problems, specifically for Yang--Mills connections and Einstein metrics.

Let $\prs{X^n,g}$ be a smooth, compact Riemannian manifold, and suppose $\N$ is a connection on a vector bundle $E$ over $X$. The Yang--Mills energy associated to $\N$ is given by
\begin{align*}
\mathcal{YM}\brk{\N} := \int_{X^n} \brs{F_{\N}}^2 \dV_g.
\end{align*}
Critical points for $\mathcal{YM}$ are called Yang--Mills connections, including the special class of instantons, which always minimize $\mathcal{YM}$ when they exist.  While there are many existence results for instantons (eg. \cite{Taubes2}), it is also known that generically one expects non-instanton, non-minimizing Yang--Mills connections to exist even in the critical dimension $n=4$ \cite{SSU, BM, SS, Bor}.  Furthermore, in dimension $4$ every stable Yang--Mills connection with small gauge group is an instanton \cite{BL}, so non-minimizing Yang--Mills connections in this setting will have positive index.  Thus, to understand the Yang--Mills functional it becomes important to understand the structure of these non-minimizing Yang--Mills connections, in particular to understand their index.  This index is that of the relevant Jacobi operator, a Schr\"odinger operator acting on Lie algebra-valued $1$-forms, with inhomogeneous term determined by the curvature of the underlying Riemannian metric as well as the bundle connection's curvature.  Taking a cue from the CLR estimate one may hope roughly that for a connection to have high index it must also have high Yang--Mills energy.  The first main result yields an estimate of this type.

\begin{thm}\label{t:mainthm} Let $(X^4,g)$ be a closed, oriented four-manifold, with Yamabe invariant $\cY (X^4,[g]) > 0$.   Suppose $\nabla$ is a non-instanton Yang--Mills connection on a vector bundle $E$ over
$X^4$ with structure group $G \subset \SO(E)$, and curvature $F_{\N}$.  Let $\imath(\N)$ denote the index and $\nu(\N)$ the nullity of $\N$.  Then
\begin{align*}
\imath(\nabla) + \nu(\nabla) &\leq  \dfrac{ 144 e^2 \dim(\mathfrak{g}_E)   } {\cY(X^4, [g])^2 } \Big\{ - 12 \pi^2 \Eul\prs{X^4} + 12 \int_{X} |F_{\N}|^2 \, \dVg   \\
& \ \ \ \ \ \ + 3 \sqrt{2} \int_X |W_{g}||F_{\N}| \, \dVg + 3 \int_{X} | W_{g} |^2 \,\dVg  \Big\},
\end{align*}
where $e$ is Euler's number, $\Eul(X^4)$ is the Euler characteristic and $W_{g}$ is the Weyl tensor.
\end{thm}

If $\nabla$ is an instanton, then $\nu(\nabla) = 0$ and the Atiyah--Singer index formula gives an explicit formula for $\imath(\nabla)$ depending on topological data (see Chapter 4 of \cite{DK}).  Our statement explicitly does not include this case, and we use the assumption of nonvanishing of $F^+_{\N}$ when constructing a metric conformal to the base, with respect to which we carry out the index estimate (see Proposition \ref{CGF}).  When the base manifold is the round sphere we can simplify the statement to the following:

\begin{cor} Let $E \rightarrow (\mathbb{S}^4,g_{\mathbb S^4})$ be a vector bundle over the round sphere with structure group $G \subset \SO(E)$, with $\N$ a non-instanton Yang--Mills connection.  Then
\begin{align*}
\imath(\nabla) + \nu(\nabla) \leq 9 e^2 \dim(\mathfrak{g}_E)  \Big\{ - 1 + \tfrac{1}{4 \pi^2} \int_{\mathbb{S}^4} |F_{\N}|^2 \,\dVg \Big\}.
\end{align*}
\end{cor}
An index plus nullity estimate for Yang--Mills connections appeared in \cite{Urakawa}, under the much stronger assumption that the base manifold has positive Ricci curvature  and with a bound depending on the $L^{\infty}$-norm of the bundle curvature.  Our result only assumes positive Yamabe invariant, and the bound depends on conformal invariants of the base manifold and the Yang--Mills energy.  This is more natural, in view of the fact that the index and nullity are conformal invariants.   Furthermore, although the constants in Theorem \ref{t:mainthm} are almost certainly not sharp (in fact, the sharp value is not known in the classical CLR inequality; cf. \cite{HKRV}), we can show by means of examples that the growth rate of the index as a function of the Yang--Mills energy is sharp.  Specifically, combining an index estimate of Taubes \cite{Taubes} as well as an explicit construction of non-instanton Yang--Mills connections due to Sadun--Segert \cite{SS}, we exhibit a family of connections whose index grows linearly in the Yang--Mills energy (Proposition \ref{p:lineargrowth} below).  Lastly we point out that the estimate we give in \S \ref{LYarg} can be adapted to give an index estimate for Yang--Mills connections in any dimension in terms of the $L^{\frac{n}{2}}$ norms of $F$ and the Ricci curvature, and the Sobolev constant, and in this case the proof is a very direct adaptation of the method of Li--Yau \cite{LY} (see Remark \ref{r:gendim}).

Our second main result is an index estimate for Einstein metrics in dimension four.  Einstein metrics arise as critical points of the normalized total scalar curvature functional
\begin{align} \label{EH}
\mathscr{S}[g] = \Vol(g)^{-1/2} \int_{X^4} R_g \,\dV_g.
\end{align}
It is well-known that Einstein metrics are never stable critical points, since $\mathscr{S}$ is minimized over conformal variations but is locally maximized over transverse-traceless variations, possibly up to a finite dimensional subspace.  The index $\imath(g)$ of an Einstein metric is dimension of the maximal subspace on which the second variation is negative when restricted to transverse traceless variations, while the nullity $\nu(g)$ is the dimension of the space of infinitesimal Einstein deformations.  While there are some works characterizing the stability and space of deformations of Einstein metrics (\cite{Koiso1, Koiso2, DWW1, DWW2}), it seems very little is known about the index in the case it is positive.  Intuitively, one might expect an Einstein metric with large index to have small energy.  We derive an estimate of this kind which relies on explicit universal constants and the Euler characteristic.

\begin{thm}  \label{t:mainthm2}  Let $(X^4,g)$ be an Einstein four-manifold with positive scalar curvature.  Then
\begin{align*}
\mathscr{S}[g] \leq 24 \pi \sqrt{ \dfrac{  \Eul(X^4)}{ 3  +  \delta \brk{ \imath(g) + \nu(g)} }},
\end{align*}
where $\delta = \frac{1}{24 e^2}$, and $e$ is Euler's number.
\end{thm}

Our final application is a bound on the Betti numbers of an oriented four-manifold $X^4$ of positive scalar curvature.  Bounds for the Betti numbers in terms of the curvature, Sobolev constant, and diameter of the manifold were proved by P. Li in \cite{PLiBetti}.  These estimates can be viewed as refined or quantitative versions of the classical vanishing theorems; see \cite{Berard} for a beautiful survey.  To state our results we need to introduce two conformal invariants of four-manifolds with positive Yamabe invariant.

To define the first conformal invariant, we need some additional notation.  Let $A = A_g$ denote the Schouten tensor of $g$:
\begin{align*}
A = \tfrac{1}{2} \prs{ \Ric - \tfrac{1}{6}R g },
\end{align*}
where $\Ric$ is the Ricci tensor and $R$ the scalar curvature of $g$.  Let $\sigma_2(A)$ denote the second symmetric function of the eigenvalues of $A$ (viewed as a symmetric bilinear form on the tangent space at each point).  Then
\begin{align*}
\sigma_2(A) = -\tfrac{1}{8}|\Ric|^2 + \tfrac{1}{24}R^2.
\end{align*}
The integral of this expression is a scalar conformal invariant of a four-manifold.  Using this we define the following two conformal invariants:
\begin{align} \label{rhodef} \begin{split}
\rho_1(X^4,[g]) &:= \dfrac{4 \int_{X} \sigma_2(A)\, \dV}{\cY(X^4,[g])^2}, \\
\rho_{+}(X^4, [g]) &:= \dfrac{ 24 \int_{X} |W^{+}|^2 \, \dV}{\cY(X^4,[g])^2}.
 \end{split}
 \end{align}

Let $b_1(X^4)$ denote the first Betti number of $X^4$, and let $b^{+}(X^4)$ denote the maximal dimension of a subspace of $\Lambda^2(X^4)$ on which the intersection form is positive.  It follows from (\cite{Gursky98} Theorem 2) that if $b_1(X^4) > 0$ then $\rho_1 \leq 0$, with equality only when conformal to a quotient of $S^3 \times \mathbb R$ with the product metric.   Furthermore, it follows from (\cite{Gursky} Theorem 3.3) that if $b^+ > 0$ then $\rho_+ \geq 1$, with equality only when conformal to a K\"ahler metric with positive scalar curvature.  Using the general index estimate of Section \ref{LYarg}, we can prove quantitative versions of these estimates:

\begin{thm} \label{t:mainthm3} Let $(X^4,g)$ be an oriented four-manifold with $\cY(X^4,[g]) > 0$.  Then
\begin{align} \label{b1}
b_1(X^4) \leq 9e^2 \prs{ 1 - 24 \rho_1},
\end{align}
and
\begin{align} \label{bplus}
b^{+}(X^4) \leq 3 e^2 \prs{ 2 \sqrt{\rho_{+}} - 1 }^2,
\end{align}
where $e$ is Euler's number.
\end{thm}

\noindent Here, as in the Yang-Mills estimate, our constants are likely not sharp but the growth rate is.  In particular, by taking connect sums with sufficiently long necks, we can produce locally conformally flat metrics on the manifold $k \# \mathbb{S}^3 \times \mathbb{S}^1$ whose Yamabe invariant is uniformly bounded below.  Evidently this manifold has $b_1 = k$, while for these conformal classes we see that the right hand side of (\ref{b1}) grows linearly in $k$.


The proofs of these theorems all rely on an extension of the CLR estimate to elliptic operators on vector bundles with certain geometric backgrounds (see Section \ref{LYarg}).   The case of dimension $n=4$ especially requires careful analysis of the curvature terms in the relevant index operator in order to capture the conformal invariance.  While many proofs of the classical CLR inequality by now exist, the proof of Li--Yau \cite{LY} gives explicit bounds in terms of the Sobolev constant.  By adapting their ideas to operators modeled on the conformal Laplacian but acting on sections of a vector bundle, we are able to obtain estimates in terms of conformal invariants.  An important technical step is to compare the $L^2$-trace of the heat kernel of a Schr\"odinger-type operator acting on sections of a vector bundle to the heat trace of an associated scalar operator.  Again, many results of this kind exist (see \cite{HSU2, HSU, Simon}), but we adapt a proof of Donnely--Li \cite{DL} as it is closest in spirit to the other estimates.  Combining these ideas together with a conformal gauge-fixing argument yields our main index estimates.
\subsection*{Acknowledgements} The authors thank Elliott Lieb, Francesco Lin, Zhiqin Lu, and Richard Schoen for informative discussions.
\section{General index estimate}\label{LYarg}

In this section we adapt the proof of the Cwikel--Lieb--Rosenblum inequality due to Li--Yau \cite{LY} to prove an index estimate for a certain class of elliptic operators acting on sections of vector bundles.  Given a vector bundle $\mathcal{E} \rightarrow (X^4, g)$ with a metric-compatible connection $\nabla$, let $\Delta = \Delta_g : \Gamma(\mathcal{E}) \rightarrow \Gamma(\mathcal{E})$ denote the rough Laplacian.  Given a non-negative function $V \in C^0(X^4)$, consider the operator
\begin{align} \label{Sdef}
\mathcal{S} = -\Delta + \tfrac{1}{6}R - V,
\end{align}
where $R = R_g$ is the scalar curvature of $g$.  We will assume throughout this section that $R \geq 0$, and the Yamabe invariant $\cY(X^4,\brk{g}) > 0$.  Our main result is

\begin{thm}  \label{NestProp}  If $N_0(\mathcal{S})$ denotes the number of non-positive eigenvalues of $\mathcal{S}$, then
\begin{align} \label{Nest}
N_0(\mathcal{S}) \leq \dfrac{ 36 e^2 \rank (\mathcal{E})  \| V \|_{L^2}^2 } {\cY\prs{X^4, [g]}^2 }.
\end{align}
\end{thm}

The proof is a consequence of a series of technical lemmas, and will appear at the end of the section.  We begin with some notation.  We need to distinguish between the Laplacian on functions and the rough
Laplacian acting on sections of $\mathcal{E}$, so from now on we set
\begin{align*}
\Delta_0 &: C^{\infty}\prs{X^4} \rightarrow C^{\infty}\prs{X^4}, \\
\Delta &:  C^{\infty}\prs{\mathcal{E}} \rightarrow C^{\infty} \prs{\mathcal{E}}.
\end{align*}
Fix some small $\epsilon > 0$ define
\begin{align} \label{Ve}
V_{\ge} := V + \epsilon.
\end{align}
Consider the two operators
\begin{align*}
\mathcal{P}_0 := \tfrac{1}{V_{\ge}} \prs{ \Delta_0 - \tfrac{1}{6}R},
\end{align*}
\begin{align*}
\mathcal{P} := \tfrac{1}{V_{\ge}}  \prs{ \Delta - \tfrac{1}{6}R}.
\end{align*}

As a first step we give the following analogue of an estimate in Li--Yau:

\begin{lemma} \label{HeatLemma}  Let $\mu_1^0 \leq \mu_2^0 \leq \cdots$ denote the eigenvalues of $-\mathcal{P}_0$, counted with multiplicity.  Then for all $t > 0$,
\begin{align} \label{ht}
\sum_{i=1}^{\infty} e^{-2\mu_i^0 t} \leq  \frac{ 36  \brs{\brs{ V_{\ge} }}_{L^2}^2 } {\cY(X^4, [g])^2 }  t^{-2}.
\end{align}
\end{lemma}

\begin{proof}  As in \cite{LY}, we take $\{ \psi_i \}$ to be an orthonormal basis of $L^2 \prs{V_{\ge} \dV}$ consisting of eigenfunctions of $-\mathcal{P}_0$:
\begin{align*}
- \mathcal{P}_0 \psi_i = \mu_i  \psi_i,
\end{align*}
with
\begin{align*}
\int_X \psi_i(x) \psi_j(x) V_{\ge}(x) \,\dV_x \equiv \delta_{ij}.
\end{align*}
Let
\begin{align*}
H_0(x,y,t) := \sum_{i=1} e^{-t \mu_i} \psi_i(x) \psi_i(y).
\end{align*}
Note that $H_0$ is the heat kernel associated to the operator $\mathcal{P}_0$ with respect to the weighted inner product $L^2(V_{\ge} \dV)$.   In particular,
\begin{align} \label{Heat} \begin{split}
\tfrac{\partial}{\partial t}\brk{H_0(x,y,t)} &= \mathcal{P}_0 H_0(x,y,t) \\
&= \tfrac{1}{V_{\ge}} \prs{ \Delta_0 - \tfrac{1}{6}R } H_0(x,y,t).
\end{split}
\end{align}
Moreover, since $R \geq 0$ we have
\begin{align*}
H_0(x,y,t) > 0,
\end{align*}
and for any $f \in C^0\prs{X^4}$,
\begin{align} \label{Delta}
\lim_{t \to 0} \int_X H_0(x,y,t) f(y) V_{\ge}(y) \,\dV_y = f(x).
\end{align}
We also let
\begin{align*}  \begin{split}
h(t) &:= \int_{X} \int_{X}  H_0(x,y, t)^2 V_{\ge}(x) V_{\ge}(y) \, \dV_x \, \dV_y \\
&= \sum_{i=1}^{\infty} e^{-2\mu_i^0 t}.
\end{split}
\end{align*}

We now argue as in the proof of Theorem 2 of \cite{LY}: differentiating $h$, using (\ref{Heat}) and integrating by parts, we have
\begin{align} \label{muterm} \begin{split}
\tfrac{d h}{d t} &= 2 \int_{X} V_{\ge}(x) \int_{X} H_0(x,y,t) (\mathcal{P}_0)_y H_0(x,y,t) V_{\ge}(y) \, \dV_y \, \dV_x \\
&= 2 \int_{X} V_{\ge}(x) \int_{X} H_0(x,y,t) \prs{ \Delta_0 - \tfrac{1}{6}R }_y H_0(x,y,t) \,\dV_y \, \dV_x.\\
&= - 2 \int_{X} V_{\ge}(x) \int_{X} \brk{ \brs{\nabla_y H_0(x,y,t)}^2 + \tfrac{1}{6}R(y) H_0(x,y,t)^2 } \,\dV_y \, \dV_x,\\
&= - 2 \int_{X} V_{\ge}(x) \int_{X} \brk{ \brs{\nabla_y H_0(x,y,t)}^2 + \tfrac{1}{6}R(y)  H_0(x,y,t)^2 } \, \dV_y \, \dV_x.
\end{split}
\end{align}
By the definition of the Yamabe invariant,
\begin{align*}
\cY(X^4,[g]) \Big( \int_{X} H_0(x,y,t)^4  \,\dV_y \Big)^{1/2} \leq 6 \int_{X} \big[ |\nabla_y H_0(x,y,t)|^2 + \tfrac{1}{6}R_y H_0(x,y,t)^2 \big] \,\dV_y.
\end{align*}
Using this, we can rewrite \eqref{muterm} as
\begin{align} \label{dhdt1}
\tfrac{d h}{d t} \leq - \tfrac{1}{3} \cY(X^4, [g]) \int_{X} V_{\ge}(x) \prs{ \int_{X} H_0(x,y,t)^4  \,\dV_y }^{1/2} \, \dV_x.
\end{align}

To obtain a differential inequality for $h$ we need a further a priori upper bound.  Iterating H\"older's inequality twice and using the fact that $H_0(x,y,t) > 0$ we note
\begin{align} \begin{split} \label{hup}
h(t) &= \int_X V_{\ge}(x) \int_X  H_0(x,y,t)^2 V_{\ge}(y) \dV_y \dV_x \\
&\leq \int_X V_{\ge}(x) \brk{ \prs{\int_X H_0(x,y,t)^4 \dV_y }^{1/3} \prs{ \int_X H_0(x,y,t) V_{\ge}^{3/2}(y) \dV_y }^{2/3} }\dV_x \\
&\leq \brk{ \int_X V_{\ge}(x) \prs{ \int_X H_0(x,y,t)^4 \dV_y }^{1/2} \dV_x }^{2/3}  \brk{ \int_X V_{\ge}(x) \prs{ \int H_0(x,y,t) V_{\ge}^{3/2}(y) \dV_y }^2 \dV_x }^{1/3},
\end{split}
\end{align}
It remains to estimate the second term on the right hand side above, which is done by treating it as an auxiliary solution to the heat equation.  In particular set
\begin{align} \label{Qdef}
Q(x,t) := \int_X  H_0(x,y,t) V_{\ge}(y)^{3/2} \,\dV_y.
\end{align}
Note that $Q$ is a solution of the heat equation associated to $\mathcal{P}_0$:
\begin{align}\label{eq:Qevolution}
\begin{split}
\tfrac{\partial}{\partial t} \brk{ Q(x,t) }&= (\mathcal{P}_0 Q)(x,t) = \tfrac{1}{V_{\ge}(x)} \prs{ \Delta_0 - \tfrac{1}{6}R  } Q(x,t),\\
Q(x,0) &= V^{1/2}_{\ge}(x).
\end{split}
\end{align}
Note in particular the power of $V_{\ge}$, which is a consequence of the weighted inner product.  We first compute
\begin{align}\label{eq:QVev}
\begin{split}
\tfrac{d}{dt} \brk{\int_X Q(x,t)^2 V_{\ge}(x) \,\dV_x} &= 2 \int_X Q(x,t) \tfrac{\partial}{\partial t}\brk{Q(x,t)} V_{\ge}(x) \,\dV  \\
&= 2 \int_X Q(x,t) \prs{ \Delta_0 - \tfrac{1}{6}R} Q(x,t) \,\dV \\
&= -2 \int_X \brk{ \brs{\nabla Q(x,t)}^2 + \tfrac{1}{6} R(x) Q(x,t)^2 } \,\dV \\
&\leq 0.
\end{split}
\end{align}
Integrating this and applying \eqref{eq:Qevolution},
\begin{align*}\begin{split}
\int_X Q(x,t)^2 V_{\ge}(x) \,\dV &\leq \int_X Q(x,0)^2 V_{\ge}(x) \,\dV \\
&= \int_X V_{\ge}(x)^2 \,\dV.
\end{split}
\end{align*}

Now, using \eqref{Qdef},
\begin{align*}
\int_X Q(x,t)^2 V_{\ge}(x) \,\dV = \int_X V_{\ge}(x) \prs{ \int_X H_0(x,y,t) V_{\ge}(y)^{3/2}\dV_y }^2 \,\dV_x,
\end{align*}
and so substituting into \eqref{eq:QVev} we obtain
\begin{align*}
\| V_{\ge} \|_{L^2} \geq \brk{ \int_X V_{\ge}(x)\prs{ \int_X H_0(x,y,t) V_{\ge}(y)^{3/2} \dV_y }^2 \, \dV_x }^{1/2}.
\end{align*}
Substituting this into \eqref{hup}, we have
\begin{align*}
h(t) \leq \brk{\int_X V_{\ge}(x) \prs{ \int_X H_0(x,y,t)^4 \dV_y }^{1/2} \dV_x }^{2/3} \| V_{\ge} \|_{L^2}^{2/3}.
\end{align*}
By \eqref{dhdt1}, we conclude
\begin{align} \label{hp1}
\frac{dh}{dt} \leq - \tfrac{1}{3} \frac{ \cY(X^4, [g]) }{ \brs{\brs{ V_{\ge} }}_{L^2}}  h(t)^{3/2}.
\end{align}
Integrating and using the fact that $h(t) \rightarrow \infty$ as $t \rightarrow 0^{+}$ we conclude
\begin{align*}
h(t) \leq  \frac{ 36 \brs{\brs{ V_{\ge} }}_{L^2}^2}{ \cY(X^4, [g])^2 }  t^{-2},
\end{align*}
which is equivalent to (\ref{ht}).
\end{proof}

The key lemma that allows us to pass from Lemma \ref{HeatLemma} to Theorem \ref{NestProp} is the following:

\begin{lemma} \label{KeyLemma}  We have
\begin{align} \label{Keycomp}
\tr e^{t \mathcal{P}}  \leq  \rank (\mathcal{E})  \tr e^{t \mathcal{P}_0}.
\end{align}
\end{lemma}
\begin{proof}  This is based on argument in \cite{DL}, Theorem 4.3 and Corollary 4.4.  Let $H(x,y,t)$ denote the heat kernel associated to $\mathcal{P}$ with respect to the weighted inner product of Lemma \ref{HeatLemma}.  More precisely, let $\mu_1 \leq \mu_2 \leq \cdots$ denote the eigenvalues of $-\mathcal{P}$, counted with multiplicity, and let $\{ \phi_i \}$ be an orthonormal basis of sections of  $L^2(\mathcal{E},V_{\ge} \dV)$ consisting of eigenfunctions of $-\mathcal{P}$:
\begin{align*}
- \mathcal{P} \phi_i = \mu_i  \phi_i,
\end{align*}
with
\begin{align*}
\int_X \ip{ \phi_i(x) , \phi_j(x) } V_{\ge}(x) \,\dV_x = \delta_{ij}.
\end{align*}
Then the associated heat kernel is given by
\begin{align*}
H \prs{x,y,t} = \sum_{i=1} e^{-t \mu_i} \phi_i \prs{x} \otimes \phi_i \prs{y}.
\end{align*}
If $|H|$ denotes the norm of $H$ as an endomorphism $H(\cdot,x,y) : \mathcal{E}_x \rightarrow \mathcal{E}_y$, then $\brs{H}$ is a subsolution of (\ref{Heat}) (in the sense of distributions):
\begin{align} \label{Heat2} \begin{split}
\tfrac{\partial}{\partial t}\brk{\brs{H}\prs{x,y,t}} &\leq \mathcal{P}_0 \brs{H}\prs{x,y,t} \\
&= \tfrac{1}{V_{\ge}} \prs{ \Delta_0 \brs{H} - \tfrac{1}{6}R }  \brs{H}\prs{x,y,t},
\end{split}
\end{align}
see Lemma 4.1 of \cite{DL}.  Also, in analogy with \eqref{Delta}, for any $f \in C^0\prs{X^4}$ we have
\begin{align}\label{Delta1}
\lim_{t \to 0} \int_X \brs{H\prs{x,y,t}} f\prs{y} V_{\ge}\prs{y} \, \dV_y = f\prs{x}.
\end{align}
By \eqref{Delta} and \eqref{Delta1},
\begin{align} \label{DP1} \begin{split}
\brs{H}\prs{x,y,t} &- H_0\prs{x,y,t} \\
&= \lim_{\tau \rightarrow 0} \sqg{ \int_{X} |H|(x,z,t) H_0(z,y,\tau) V_{\ge}(z) \,\dV_z - \int_{X} \brs{H}\prs{x,z,\tau} H_0\prs{z,y,t} V_{\ge}\prs{z} \, \dV_z  } \\
&= \int_0^t \tfrac{d}{ds} \brk{\int_{X} \brs{H}\prs{x,z,s} H_0\prs{z,y,t-s} V_{\ge}(z) \,\dV_z} \ds \\
&= \brk{\int_0^t \brk{\int_{X} \tfrac{d}{ds}\brk{\brs{H}\prs{x,z,s}} H_0\prs{z,y,t-s} V_{\ge}(z) \,\dV_z} \ds}_{T_1}\\
&\qquad +  \brk{ \int_0^t \brk{\int_{X} \brs{H}\prs{x,z,s}\tfrac{d}{ds}\brk{ H_0\prs{z,y,t-s}} V_{\ge}(z) \,\dV_z} \ds }_{T_2}.
\end{split}
\end{align}
We manipulate the second term $T_2$ using \eqref{Heat},
\begin{align} \label{DP2} \begin{split}
T_2
&= \int_0^t \int_{X} \brs{H}\prs{x,z,s} \tfrac{\partial}{\partial s} H_0\prs{z,y,t-s} V_{\ge}(z) \, \dV_z \ds \\
&=- \int_0^t \int_{X}  \brs{H}(x,z,s) \tfrac{1}{V_{\ge}(z)} \prs{ \Delta_0 - \tfrac{1}{6}R } H_0\prs{z,y,t-s} V_{\ge}(z) \,\dV_z \ds\\
&= - \int_0^t \int_{X}  \brs{H}\prs{x,z,s}  \Delta_0 H_0\prs{z,y,t-s}  \,\dV_z \ds + \int_0^t \int_{X} \tfrac{\partial}{\partial s} \brs{H}\prs{x,z,s} \tfrac{1}{6}R\prs{z} H_0\prs{z,y,t-s} \,\dV_z \ds.
\end{split}
\end{align}
Integrating by parts in the term involving $\Delta_0$ and using \eqref{Heat2}, reincorporating $T_2$ into \eqref{DP1},
\begin{align} \label{DP3} \begin{split}
|H|(x,y,t) - H_0(x,y,t) &= \int_0^t \int_{X} \prs{ \tfrac{\partial}{\partial s}  - \mathcal{P}_0 } |H|(x,z,s) H_0(z,y,t-s) q_{\ge}(z) \,\dV_z \leq 0.
\end{split}
\end{align}
Therefore, if $\mbox{tr}_g H$ denotes the pointwise trace of $H(\cdot,x,x) : \mathcal{E}_x \rightarrow \mathcal{E}_x$,
\begin{align*}
\tr e^{t \mathcal{P}} &= \int_{X} \mbox{tr}_g H(x,x,t) V_{\ge}(x) \,\dV_x \\
&\leq  \mbox{rank}(\mathcal{E}) \int_{X} |H|(x,x,t) V_{\ge}(x) \,\dV_x \\
&\leq  \mbox{rank}(\mathcal{E}) \int_{X} H_0(x,x,t) V_{\ge}(x) \,\dV_x \\
&= \mbox{rank}(\mathcal{E})\ \tr e^{t \mathcal{P}_0}.
\end{align*}
The result follows.
\end{proof}

Combining Lemma \ref{HeatLemma} with Lemma \ref{KeyLemma} we have
\begin{prop} \label{comprop} Let $\mu_1 \leq \mu_2 \leq \cdots$ denote the eigenvalues of $-\mathcal{P}$, counted with multiplicity.  Then for all $t > 0$,
\begin{align} \label{ht1E}
\sum_{i=1}^{\infty} e^{-2\mu_i t} \leq  \frac{ 36  \rank (\mathcal{E})  \brs{\brs{ V_{\ge} }}_{L^2}^2 } {\cY(X^4, [g])^2 }  t^{-2}.
\end{align}

\begin{proof}
Observe that
\begin{align} \label{2t}
\sum_{i=1}^{\infty} e^{-2\mu_i t} = \prs{ \tr e^{\prs{2t} \mathcal{P}} }.
\end{align}
But by Lemma \ref{KeyLemma},
\begin{align} \label{htr}
\big( \tr e^{(2t) \mathcal{P}}\big)  &\leq \rank (\mathcal{E})\ \big( \tr e^{(2t) \mathcal{P}_0}\big) = \rank (\mathcal{E}) \sum_{i=1}^{\infty} e^{-2\mu_i^0 t}.
\end{align}
Thus the result follows from Lemma \ref{HeatLemma}.
\end{proof}
\end{prop}



\begin{cor}  \label{Thm2Prop} Let $\mu_k$ denote the $k^{th}$-eigenvalue of $-\mathcal{P}$. Then
\begin{align} \label{muk}
\dfrac{ 36 e^2 \rank (\mathcal{E})   \| V_{\ge} \|_{L^2}^2 } {\cY(X^4, [g])^2 } \mu_k^2 \geq k.
\end{align}
\end{cor}

\begin{proof}  As in \cite{LY}, take $t = \frac{1}{\mu_k}$ in \eqref{ht1E}, then
\begin{align*}
\dfrac{ 36 \rank (\mathcal{E})   \| V_{\ge} \|_{L^2}^2 } {\cY(X^4, [g])^2 } \mu_k^2 &\geq \sum_{i=1}^{\infty} \exp(- 2 \tfrac{\mu_i}{\mu_k}) \\
&\geq \sum_{i=1}^{k} \exp(- 2 \tfrac{\mu_i}{\mu_k}) \\
&\geq k e^{-2}.
\end{align*}
The result follows.
\end{proof}

\medskip

\begin{proof}[Proof of Theorem \ref{NestProp}]  By the argument of Birman--Schwinger, the number of non-positive eigenvalues of the operator $-\Delta + \frac{1}{6}R + V_{\ge}$  is less than or equal to the number of eigenvalues of the operator $-\mathcal{P} = \frac{1}{V_{\ge}} (-\Delta + \frac{1}{6} R)$ that are less than or equal to $1$.   But by (\ref{muk}), if $\mu_k$ the greatest eigenvalue of $-\mathcal{P}$ that is less than or equal to $1$, then
\begin{align*} \begin{split}
k 
&\leq \dfrac{ 36 e^2 \rank (\mathcal{E}) \| V_{\ge} \|_{L^2}^2 } {\cY(X^4, [g])^2 }.
\end{split}
\end{align*}
Therefore, taking $\epsilon \rightarrow 0$ we conclude
\begin{align*}
N_0(\mathcal{S}) \leq \dfrac{ 36 e^2 \rank (\mathcal{E}) \| V \|_{L^2}^2 } {\cY(X^4, [g])^2 },
\end{align*}
which completes the proof. \end{proof}

\begin{rmk} \label{r:gendim} If $\mathcal{E} \rightarrow (X^n, g)$ is a vector bundle, $n \geq 3$, and $\mathcal{S} = -\Delta + V$ is a linear operator acting on sections of $E$ with $V \geq 0$, then the preceding arguments can easily be adapted to give an estimate for the number of non-positive eigenvalues of $\mathcal{S}$.  If $C_S(g)$ denotes the Sobolev constant,
\begin{align*}
C_S(g) \prs{ \int_{X} |f|^{\frac{2n}{n-2}} \, \dV }^{\frac{n-2}{n}} \leq \int_X \brk{ |\nabla f|^2 + f^2 } \, \dV,
\end{align*}
then
\begin{align*}
N_0(\mathcal{S}) \leq c_n \dfrac{ \rank (\mathcal{E}) }{ C_S(g)^{\frac{n}{2}}}  \| (1 + V) \|_{L^{n/2}}^{n/2}.
\end{align*}
\end{rmk}

\section{Index estimate for Yang-Mills connections}
\subsection{Background}  \label{conSec}

Let $(E,h) \to (X^n,g)$ be a vector bundle with metric over a closed Riemannian manifold with structure group $G \subset \SO(E)$. Let
$\Gamma(E)$ denote the smooth
sections of $E$, and $\mathfrak{g}_E$ denote the associated Lie algebra of $E$. For each point $x \in X^n$ choose a local orthonormal basis of
$TX^n$ given by $\{ e_i \}$ with dual basis $\{ e^i \}$ and a local basis for    
$E$ given by $\{ \mu_{\ga} \}$ with dual basis $\{ (\mu^*)^{\ga} \}$ of the
dual $E^*$.   Let $\Lambda^p$ denote the space of smooth $p$-forms over $X$ and set $\Lambda^p(E)
:= \Lambda^p \otimes \Gamma(E)$.  Given an element in $\Lambda^p(E)$ its components are understood be with respect to the forgoing bases.  We will also use the fact that when $p=1$, we can take tensor products of the basis elements $\{e^i\}, \{ \mu_{\ga} \}, \{ (\mu^*)^{\ga} \}$ to obtain a (local) basis of $\Lambda^1(E)$.

We will use the following conventions for the various inner products that appear:
\begin{align*}
\ip{\eta,\gw}_{\Lambda^2} &= \tfrac{1}{2} \sum_{i,j} \eta_{ij} \gw_{ij}, \qquad \ip{\nu,\mu}_{S^2_0(X)} =  \sum_{i,j}\nu_{ij} \mu_{ij}, \\
\ip{A,B}_{\mathfrak{g}_E} &= - \tfrac{1}{2} \trace_E \prs{AB} = - \tfrac{1}{2} \sum_{\ga,\gb} A^{\gb}_{\ga} B^{\ga}_{\gb}\\
\ip{P, Q}_{\Lambda^1 \prs{\mathfrak{g}_E}} &= - \tfrac{1}{2} P_{i \gb}^{\ga} Q_{i \ga}^{\gb}, \qquad \ip{R, S}_{\Lambda^2 \prs{\mathfrak{g}_E}} = - \tfrac{1}{4} R_{ij \gb}^{\ga} S_{ij \ga}^{\gb}.
\end{align*}
Here, repeated Latin indices indicate contractions by the metric $g$ on $X^n$, and the components are with respect to the orthonormal basis above.  Unless specified otherwise, we will use Einstein summation notation for both bundle and base components.

We need certain algebraic actions as well.  First there is the bracket operation $[,] : \Lambda^1 (\mathfrak{g}_E) \times \Lambda^1 (\mathfrak{g}_E) \to \Lambda^2(\mathfrak{g}_E)$ defined by
\begin{align*}
\brk{A, B}_{jk \ga}^{\gb} &:= A_{j \gd}^{\gb} B_{k \ga}^{\gd} - B_{k\gd}^{\gb} A_{j \ga}^{\gd}, \qquad A, B \in \Lambda^1 (\mathfrak{g}_E).
\end{align*}
Also, given $\eta \in S^2 \prs{TX}$ and $\Phi \in \Lambda^2 \prs{\mathfrak{g}_E}$, we may view both as elements of $\End (\Lambda^1(\mathfrak{g}_E))$ via the formulas
\begin{align*}
\prs{\eta\prs{A}}_{i \ga}^{\gb} &= \eta_{ij} A_{j\ga}^{\gb}, \\
\prs{ \brk{\Phi, A} }_{i \ga}^{\gb} &=  \brk{\Phi_{ji}, A_j}_{\ga}^{\gb} = \Phi_{ji \mu}^{\gb} A_{j \ga}^{\mu} - A_{j \mu}^{\gb} \Phi_{ji \ga}^{\mu}.
\end{align*}
We next recall the definition of the Jacobi operator of $\mathcal{YM}$.

\begin{customthm}{6.8 of \cite{BL}}\label{BL6.8}  \label{JacobiThm} Suppose $\nabla$ is a Yang--Mills connection on  a vector bundle $E$ over $X^n$ with structure group $G \subset \SO(E)$, and $\sqg{\N_s}$ is a one parameter family of connections with $\N \equiv \left. \N_s \right|_{s=0}$. Furthermore, suppose $B := \left. \tfrac{\del}{\del s} \brk{\N_s} \right|_{s=0} \in \Lambda^1(\mathfrak{g}_E)$. Then
\begin{align*}
\left. \tfrac{d^2}{d s^2 }\brk{\mathcal{YM}\prs{\N_s}} \right|_{s=0} &= 2 \int_X \ip{\mathcal{J}^{\N} \prs{B}, B}_{\Lambda^1(\mathfrak{g}_E)} \dV,
\end{align*}
where
\begin{align*}
\begin{split}
\mathcal{J}^{\N} \prs{B}_i
&= -\lap B_i - \N_i \N^j B_j + 2 \brk{F_{ji}, B^j} + \Ric_i^j B_j,
\end{split}
\end{align*}
where $\lap = \N^a \N_a$ denotes the rough Laplacian.
\end{customthm}
The operator $\mathcal{J}^{\N}$ is degenerate elliptic, due to the action of the infinite dimensional gauge group.  Questions of index and nullity always refer to the operator restricted to divergence-free sections $B$, one which the operator takes the simpler form:
\begin{align}
\begin{split}\label{eq:Jacobioperator2}
\mathcal{J}^{\N} \prs{B}_i &= -\lap B_i + 2 \brk{F_{ji}, B^j} + \Ric_i^j B_j.
\end{split}
\end{align}
The index and nullity of a Yang-Mills connection are understood to be those quantities associated to this operator.  It follows from the conformal invariance of the Yang-Mills energy that both the index and nullity are conformally invariant.

\subsection{Linear algebraic estimates}

In this subsection we obtain linear algebraic estimates which enter into estimating the Jacobi operator.  The key point is Proposition \ref{Bsharp}, which provides a sharp inequality between the operator and Hilbert-Schmidt norms of the bilinear form appearing in the Jacobi operator.  Let $\cZ \in S_0^2 \prs{TX}$ and $\Phi \in \Lambda^2 \prs{\mathfrak{g}_E}$; in the following we can view both as elements of $\End (\Lambda^1(\mathfrak{g}_E))$.

\begin{lemma} \label{Bsymm}  Suppose $E \to \prs{X^n,g}$ is a vector bundle. Then $\cZ$ and $\Phi$, viewed as endomorphisms of $\Lambda^1\prs{\mathfrak{g}_E }$, are symmetric.  Moreover, $\cZ$ is trace-free as an endomorphism of $\Lambda^1(\mathfrak {g}_E)$.
\begin{proof} Take $A,B \in \Lambda^1 (\mathfrak{g}_E)$. Using the symmetry of both $\cZ$ and the inner product on $E$,
\begin{align*}
\ip{\cZ \prs{A}, B}_{\Lambda^1 (\mathfrak{g}_E)} &=- \tfrac{1}{2} \cZ_{ij} A_{j\ga}^{\gb} B_{i \gb}^{\ga} \\
&= - \tfrac{1}{2} \cZ_{ji} B_{i \ga}^{\gb} A_{j\gb}^{\ga} \\
&= \ip{\cZ \prs{B}, A}_{\Lambda^1 (\mathfrak{g}_E)}.
\end{align*}
The symmetry of $\cZ$ follows. Next, using the cyclicity of inner products over $\mathfrak{g}_E$, reindexing and skew symmetry of the bracket operation and $\Phi$,
\begin{align*}
\ip{ \brk{\Phi,A}, B}_{\Lambda^1 (\mathfrak{g}_E)} &= - \tfrac{1}{2} \brk{\Phi_{ij}, A_i}^{\gb}_{\ga} B_{j \gb}^{\ga} \\
&= - \tfrac{1}{2} \Phi_{ij \gd}^{\gb} A_{i \ga}^{\gd} B_{j \gb}^{\ga} + \tfrac{1}{2}  A_{i \gd}^{\gb} \Phi_{ij \ga}^{\gd} B_{j \gb}^{\ga} \\
&= - \tfrac{1}{2} A_{i \ga}^{\gd} B_{j \gb}^{\ga}\Phi_{ij \gd}^{\gb}  + \tfrac{1}{2}  A_{i \gd}^{\gb} \Phi_{ij \ga}^{\gd} B_{j \gb}^{\ga}\\
&= - \tfrac{1}{2} A_{i \gd}^{\gb} B_{j \ga}^{\gd}\Phi_{ij \gb}^{\ga}  + \tfrac{1}{2}  A_{i \gd}^{\gb} \Phi_{ij \ga}^{\gd} B_{j \gb}^{\ga}\\
&= - \tfrac{1}{2} A_{i \gd}^{\gb} \brk{B_{j}, \Phi_{ij}}_{\gb}^{\gd}\\
&= - \tfrac{1}{2} A_{i \gd}^{\gb} \brk{\Phi_{ji},B_{j}}_{\gb}^{\gd}\\
&= \ip{ \brk{\Phi,B}, A}_{\Lambda^1 (\mathfrak{g}_E)},
\end{align*}
hence $\Phi$ is symmetric as an endomorphism.
 
To show that $\cZ$ is trace-free as an operator on $\Lambda^1(\mathfrak{g}_E)$, we construct an orthonormal basis for $\Lambda^1(\mathfrak{g}_E)$ as described at the beginning of Section \ref{conSec}: for fixed $(k,\ga,\gb)$, let
\begin{equation}\label{eq:Adefn}
A_{(k,\ga, \gb)} := e^k \ten \prs{\mu^*}^{\ga} \ten \mu_{\gb}, 
\end{equation}
where $\sqg{e_i}$ is a basis of $TM$ that diagonalizes $\cZ$. Note that the components of these basis elements are given by 
\begin{align*}
\prs{A_{(k,\ga, \gb)}}_{\ell \mu}^{\nu} = \delta_{k \ell} \delta_{\alpha}^{\nu} \delta_{\mu}^{\gb}, \qquad \ga \neq \gb,
\end{align*}
so the only nonzero entry is the $(k,\ga, \gb)$-component. Computing the trace of $\cZ$ with respect to this basis yields
\begin{align*}
\ip{\cZ (A_{(k,\ga, \gb)}) , A_{(k,\ga, \gb)} }_{\Lambda^1 (\mathfrak{g}_E)} &= - \tfrac{1}{2} \cZ_{ij} \prs{A_{(k,\ga, \gb)}}_{i \mu}^{\nu} \prs{A_{(k,\ga, \gb)}}_{j \nu}^{\mu} \\
&=  - \tfrac{1}{2} \cZ_{ij} \delta_{k i} \delta_{\alpha}^{\nu} \delta_{\mu}^{\gb} \delta_{k j} \delta_{\alpha}^{\mu} \delta_{\nu}^{\gb} \\
&=  - \tfrac{1}{2} \cZ_{ii} \delta_{\alpha}^{\gb } \delta_{\ga}^{\gb}\\
&= 0,
\end{align*}
since $\cZ$ is traceless on $TM$. The result follows.
\end{proof}
\end{lemma}

\begin{lemma}\label{lem:PhiZorthog}
As operators on $\Lambda^1(\mathfrak{g}_E)$, the ranges of $\cZ$ and $\Phi$ are orthogonal subspaces.
\begin{proof}
The orthogonality of $\cZ$ and $\brk{\Phi,\cdot}$ will follow since $\cZ$ preserves the bundle components 
 while $\Phi$ is skew symmetric with respect to the bundle components.  Using the basis \eqref{eq:Adefn} as above, for fixed $(k,\alpha, \beta)$, then
\begin{align*}
\prs{\cZ \prs{A_{(k,\alpha, \beta)}}}_{i \mu}^{\nu} &=\cZ_{\ell i} \prs{A_{(k,\alpha, \beta)}}_{\ell \mu}^{\nu}\\
&=\cZ_{\ell i} \delta_{k \ell} \delta_{\alpha}^{\nu} \delta_{\mu}^{\gb}\\
&= \cZ_{k i} \delta_{\alpha}^{\nu} \delta_{\mu}^{\gb}\\
&= \begin{cases}
\cZ_{k i} & \text{ if } \mu = \ga, \gb = \nu, \\
0 & \text{ otherwise.}
\end{cases}.
\end{align*}
Similarly,
\begin{align*}
\brk{\Phi,A_{(k,\alpha, \beta)}}_{i \mu}^{\nu} & = \Phi_{\ell i \gd}^{\nu} \prs{A_{(k,\alpha, \beta)}}^{\gd}_{\ell \mu} - \prs{A_{(k,\alpha, \beta)}}^{\nu}_{\ell \gd} \Phi^{\gd}_{\ell i \mu}  \\
&=  \Phi_{\ell i \gd}^{\nu}  \delta_{k \ell} \delta_{\alpha}^{\delta} \delta_{\mu}^{\gb} -  \delta_{k \ell} \delta_{\alpha}^{\nu} \delta_{\gd}^{\gb} \Phi^{\gd}_{\ell i \mu} \\
&=  \Phi_{k i \ga}^{\nu} \delta_{\mu}^{\gb} - \delta_{\alpha}^{\nu} \Phi^{\gb}_{k i \mu} \\
&= \begin{cases}
0 & \ga = \nu \text{ and } \gb = \mu \\
-\Phi_{k i \mu}^{\gb} & \ga = \nu \text{ and } \gb \neq \mu \\
\Phi^{\nu}_{k i \ga} & \ga \neq \nu \text{ and } \gb = \mu \\
0 & \ga \neq \nu \text{ and } \gb \neq \mu
\end{cases}.
\end{align*}
Where here, we are noting that since $\Phi \in \Lambda^2 (\mathfrak{g}_E)$, its endomorphism indices cannot coincide.
\end{proof}
\end{lemma}

To state our next result, we need to introduce an algebraic invariant defined by Bourguignon--Lawson.  Let
\begin{align*}
\gamma_{0} :=&\ \sup_{A,B \in \gG(\mathfrak{g}_E) \backslash \{0\}}
\tfrac{\brs{[A,B]}}{\brs{A} \brs{B}}.
\end{align*}
Lemma 2.30 of \cite{BL} gives the universal upper bound
\begin{align} \label{g02}
\gg_0 \leq \sqrt{2},
\end{align}
and characterizes the case of equality.

\begin{lemma} \label{SharpBracket} If $A \in \Lambda^1 (\mathfrak{g}_E)$, then
\begin{align}
\begin{split}\label{eq:AlambdaA,A}
\brs{\brk{A,A}}_{\Lambda^2(\mathfrak{g}_E)} &\leq \gg_0 \sqrt{ \tfrac{n-1}{2n}} |A|^2_{\Lambda^1(\mathfrak{g}_E)},
\end{split}
\end{align}
Since $\gg_0 \leq \sqrt{2}$, in general we have
\begin{align} \label{Brack1}
\brs{\brk{A,A}}_{\Lambda^2(\mathfrak{g}_E)} &\leq  \sqrt{ \tfrac{n-1}{n}} |A|^2_{\Lambda^1(\mathfrak{g}_E)}.
\end{align}
%
%
\end{lemma}

\begin{proof}
Fix a point $p \in X^n$ and let $\sqg{e^i}$ to be an orthonormal basis of $\Lambda^1$.  If $A \in \Lambda^1(\mathfrak{g}_E)$, then we can express $A = A_i e^i$ for $A_i \in \Gamma \prs{\mathfrak{g}_E}$.  Then
\begin{align*}
\brs{\brk{A,A}}^2_{\Lambda^2(\mathfrak{g}_E)} &= - \tfrac{1}{4} \brk{A,A}_{ij \ga}^{\gb} \brk{A,A}_{ij \beta}^{\ga} \\
&=\tfrac{1}{2} \sum_{i,j} \brs{\brk{A,A}_{ij}}^2_{\mathfrak{g}_E}\\
&= \tfrac{1}{2} \sum_{i,j} \brs{\brk{A_i,A_j}}^2_{\mathfrak{g}_E}\\
&= \sum_{i < j} \brs{\brk{A_i,A_j}}^2_{\mathfrak{g}_E}.
\end{align*}
By the definition of $\gamma_0$, this gives
\begin{align} \label{gotz} \begin{split}
\brs{\brk{A,A}}^2_{\Lambda^2(\mathfrak{g}_E)} &= \prs{ \sum_{ i < j } \brs{\brk{A_i, A_j}}^2_{\mathfrak{g}_E} } \\
&\leq \gg_0^2 \prs{\sum_{\leq i < j} \brs{A_i}^2_{\mathfrak{g}_E} \brs{A_j}^2_{\mathfrak{g}_E} }.
\end{split}
\end{align}

Now
\begin{align*}
|A|^4_{\Lambda^1(\mathfrak{g}_E)} = \sum_{i, j} \brs{A_i}^2_{\mathfrak{g}_E} \brs{A_j}^2_{\mathfrak{g}_E} &= 2 \sum_{i < j} \brs{A_i}^2_{\mathfrak{g}_E} \brs{A_j}_{\mathfrak{g}_E} ^2 + \sum_{i } \brs{A_i}^4_{\mathfrak{g}_E} ,
\end{align*}
while the arithmetic-geometric mean implies
\begin{align*}
\sum_{i } \brs{A_i}^4_{\mathfrak{g}_E}  \geq \tfrac{1}{n} \prs{\sum_{i} \brs{A_i}^2_{\mathfrak{g}_E}  }^2 = \tfrac{1}{n} |A|^4_{\Lambda^1(\mathfrak{g}_E)} .
\end{align*}
Therefore,
\begin{align*}
\sum_{1 \leq i < j \leq n} \brs{A_i}^2_{\mathfrak{g}_E}  \brs{A_j}^2_{\mathfrak{g}_E}  \leq \tfrac{(n-1)}{2n} |A|^4_{\Lambda^1(\mathfrak{g}_E)} .
\end{align*}
Substituting this into (\ref{gotz}) gives
\begin{align*}
\brs{\brk{A,A}}^2_{\Lambda^2(\mathfrak{g}_E)} &\leq \gg_0^2 \prs{ \tfrac{n-1}{2n}} |A|^4_{\Lambda^1(\mathfrak{g}_E)},
\end{align*}
and taking the square root yields \eqref{eq:AlambdaA,A}.
\end{proof}

\begin{prop} \label{Bsharp}  Suppose $E \to \prs{X^n,g}$ is a vector bundle and let
\[\mathcal{B} = \cZ + \brk{\Phi, \cdot}: \Lambda^1(\mathfrak{g}_E) \rightarrow \Lambda^1(\mathfrak{g}_E). \]
Then
\begin{align*}
\brs{ \mathcal{B} \prs{A,A}} \leq \sqrt{ \tfrac{n-1}{n}} \cdot \prs{\sqrt{ |\cZ|^2_{S_0^2(T^*M)} + 2 \gamma_0^2 |\Phi|^2_{\Lambda^2(\mathfrak{g}_E)} }} |A|^2_{\Lambda^1(\mathfrak{g}_E)}.
\end{align*}

\begin{proof}

Since $\mathcal{B}$ is symmetric by Lemma \ref{Bsymm}, there exists an orthonormal basis of $\Lambda^1(\mathfrak{g}_E)$ with respect to which the matrix of $\mathcal{B}$ is diagonalized.  Since the ranges of $\cZ$ and $\Phi$ are orthogonal by Lemma \ref{lem:PhiZorthog}, we can express the matrix of $\mathcal{B}$ as
\begin{align*}
\brk{\mathcal{B}} = \begin{pmatrix} \vec{z} & 0 \\
0 & \vec{\phi}
\end{pmatrix},
\end{align*}
where
\begin{align*}
\brk{\cZ} = \begin{pmatrix} \vec{z} & 0 \\
0 & 0
\end{pmatrix}, \ \ \ \
\brk{\Phi} = \begin{pmatrix} 0 & 0 \\
0 & \vec{\phi}
\end{pmatrix},
\end{align*}
are the matrices of $\cZ$ and $\Phi$ with respect to this basis, $\vec{z} = \prs{z_1, \cdots, z_n}$,  $\vec{\phi} = \prs{\phi_{1}, \cdots, \phi_{N} }$ are the eigenvalues of $\cZ$ and $\Phi$ respectively.  If $A \in \Lambda^1(\mathfrak{g}_E)$, then we can write $A = A_1 + A_2$, where
\begin{align*}
A_1 = \begin{pmatrix} \vec{a}\\
0
\end{pmatrix}, \ \ \
A_2 = \begin{pmatrix} 0 \\
\vec{b}
\end{pmatrix},
\end{align*}
with $\vec{a} = \prs{a_1, \cdots, a_n}$, $\vec{b} = \prs{b_{1}, \cdots, b_{N} }$.  Therefore, as a bilinear form
\begin{align*}
\mathcal{B} \prs{A,A} &= \cZ\prs{A_1,A_1} + \Phi(A_2,A_2) \\
&= \begin{pmatrix} \vec{z} & 0 \\
0 & \vec{\phi}
\end{pmatrix}\begin{pmatrix} \vec{a}\\
\vec{b}
\end{pmatrix} \cdot \begin{pmatrix} \vec{a} & \vec{b}
\end{pmatrix} \\
&= \sum_i z_i a_i^2 + \sum_{j} \phi_j b_j^2.
\end{align*}

Since $\cZ$ is trace-free via Lemma \ref{Bsymm}, 
\begin{align*}
\brs{ \cZ(A_1,A_1)} &= \brs{ \sum_i z_i a_i^2 }\\
&\leq \sqrt{\tfrac{n-1}{n}} |\vec{z}||\vec{a}|^2 \\
&= \sqrt{\tfrac{n-1}{n}}\brs{\cZ}_{S_0^2 \prs{T^*M}} |A_1|^2_{\Lambda^1(\mathfrak{g}_E)}.
\end{align*}
Also, by Lemma \ref{SharpBracket},
\begin{align*}
\brs{ \Phi(A_2,A_2)} &= \brs{ \sum_{j} \phi_j b_j^2  }\\
&= \brs{ \langle \brk{ \Phi, A_2}, A_2 \rangle }_{\Lambda^1 \prs{\mathfrak{g}_E}}\\
&= 2 \brs{ \langle \Phi, \brk{A_2,A_2} \rangle }_{\Lambda^2(\mathfrak{g}_E)} \\
&\leq 2 |\Phi |_{\Lambda^2(\mathfrak{g}_E)} \brs{ \brk{A_2,A_2}}_{\Lambda^2(\mathfrak{g}_E)} \\
&\leq2 \gamma_0 |\Phi |_{\Lambda^2(\mathfrak{g}_E)} \sqrt{ \tfrac{n-1}{2n}} \brs{A_2}_{\Lambda^1(\mathfrak{g}_E)}^2.
\end{align*}
Therefore,
\begin{align*}
\brs{ \mathcal{B}\prs{A,A}} \leq \sqrt{ \tfrac{n-1}{n}} \prs{ |\cZ|_{S_0^2 \prs{T^*M}}\brs{A_1}^2_{\Lambda^1 \prs{\mathfrak{g}_E}} + \sqrt{2}\gamma_0  |\Phi |_{\Lambda^2 \prs{\mathfrak{g}_E}} \brs{A_2}^2_{\Lambda^1 \prs{\mathfrak{g}_E}} },
\end{align*}
where we have dropped the subscripts designating the norms in order to simplify notation.  By the Cauchy-Schwartz inequality,
\begin{align*}
\brs{ \mathcal{B}\prs{A,A}} &\leq \sqrt{ \tfrac{n-1}{n}} \cdot \sqrt{ |\cZ|^2_{S_0^2(T^*M)} + 2 \gamma_0^2 |\Phi|^2_{\Lambda^2(\mathfrak{g}_E)} } \cdot \sqrt{ |A_1|_{\Lambda^1(\mathfrak{g}_E)}^4 + |A_2|^4_{\Lambda^1(\mathfrak{g}_E)} } \\
&\leq \sqrt{ \tfrac{n-1}{n}} \cdot \prs{\sqrt{ |\cZ|^2_{S_0^2(T^*M)} + 2 \gamma_0^2 |\Phi|^2_{\Lambda^2(\mathfrak{g}_E)} }} |A|^2_{\Lambda^1(\mathfrak{g}_E)}.
\end{align*}
The result follows.
\end{proof}
\end{prop}

\subsection{A canonical conformal representative} \label{Sec1}

Since the index and nullity of a Yang--Mills connection in four dimensions are conformally invariant, we may estimate them with respect to any metric conformal to the base metric $g$.   In this subsection, we specify a choice of conformal metric based on our work in \cite{GKS}.  To this end, suppose $\nabla$ is a Yang--Mills connection on a vector bundle $E$ over
$(X^4,g)$ with structure group $G \subset \SO(E)$, and denote the curvature by $F = F_{\N}$.   For $t \geq 0$, define
\begin{align*}
\Phi_{g}^t = R_{g}  - t \big[ 2 \sqrt{6} |W|_{g} + 3  \gamma_1  |F|_{g}\big],
\end{align*}
where $R_g$ is the scalar curvature of $g$, $W_g$ is the Weyl tensor, and $\gamma_1(E)$ is the constant given by
\begin{align} \label{g1Def}
\gamma_1(E) :=&\ \sup_{\omega \in \Lambda^2_{+}(\mathfrak{g}_E) \setminus \{0\}}
\dfrac{ \langle \omega, [\omega, \omega]\rangle }{|\omega|^3}.
\end{align}

\smallskip

\begin{rmk}  \label{conjob} The definition of the inner product on $\Lambda^2_{+}(\mathfrak{g}_E)$ given in \cite{GKS} differs from the definition of this paper.  In particular, the estimate for $\gamma_1(E)$ in Section 2 of \cite{GKS} needs to be adjusted.  With respect to our current conventions, we have the estimate
\begin{align} \label{g1} \begin{split}
\gamma_1(E) &\leq \tfrac{2\sqrt{6}}{3} \gamma_0(E) \\
&\leq \tfrac{4 \sqrt{3}}{3}.
\end{split}
\end{align}
\end{rmk}

\smallskip

We also define the associated operator
\begin{align*}
L^t_g = - 6 \Delta_g + \Phi_g^t.
\end{align*}

In \cite{GKS}, based on the ideas of \cite{Gursky}, we defined the related curvature and operator
\begin{align}  \label{phinow} \begin{split}
\Phi_{g} &= R_{g}  - 2 \sqrt{6} |\Wp|_{g} - 3  \gamma_1  |\Fp|_{g}, \\
L_g &= - 6 \Delta_g + \Phi_g.
\end{split}
\end{align}
It is easy to see that the expression $\gamma_1(E) |F|$ is independent of the choice of norms.  Therefore, despite the difference of conventions pointed out in Remark \ref{conjob}, the definition of $\Phi_g$ in (\ref{phinow}) agrees with the corresponding formula (3.5) in \cite{GKS}.

Observe that
\begin{align*}
\Phi_g^0 &= R_g, \\
\Phi_g^1 &\leq \Phi_g.
\end{align*}
In addition, $\Phi^t$ satisfies the same kind of conformal transformation
formula as $\Phi$: given $\hg = u^2 g$,
\begin{align*}
\Phi_{\hg}^t &= u^{-3} L_g^t u.
\end{align*}
If $\lambda_1(L^t)$ denotes the first eigenvalue of $L^t$,
\begin{align} \label{LL}
\lambda_1(L_g^t) = \inf_{\phi \in C^{\infty}(X) \backslash \sqg{0}} \dfrac{ \int_X \phi L_g^t
\phi\ \dV_g }{\int_X \phi^2\ \dV_g},
\end{align}
then the sign of $\lambda_1(L^t)$ is a conformal invariant (see \cite{Gursky}, Proposition 3.2).  In particular, by
using an eigenfunction associated with $\lambda_1(L^t)$ as a conformal factor, it
follows that $[g]$ admits a metric $\hg$ with $\Phi_{\hg}^t > 0$ ({\em resp.}, $=
0, < 0$) if and only if $\lambda_1(L_g^t) > 0$ ({\em resp.} $= 0, < 0$).

\begin{prop}  \label{CGF}  Assume $(X^4,[g])$ has  $\cY(X^4,[g]) > 0$.  Given $\nabla$ a Yang-Mills connection which is not an instanton, there exists $t_0 \in (0,1]$ such that $\lambda_1(L_g^{t_0}) = 0$.  In particular, we can choose a conformal
metric $\hg \in [g]$ with respect to which $\Phi^{t_0}_{\hg} \equiv 0$, hence
\begin{align} \label{Rprop}
R_{\hg} = 2\sqrt{6} t_0 |W_{\hg}| + 3  \gamma_1 t_0 |F|_{\hg}.
\end{align}
Moreover,
\begin{align} \label{tb}
\dfrac{ \cY(X^4,[g]) }{  2 \sqrt{6} \| W \|_{L^2} + 3 \gamma_1 \| F \|_{L^2}}\leq t_0 \leq 1.
\end{align}
\end{prop}

\begin{proof} Using the Bochner formula for Yang-Mills connections, in \cite{GKS} we showed that either $F^{+} \equiv 0$, or else $\lambda_1(L_g) = \lambda_1(L_g^1) \leq 0$. Since we are ruling out the former by assumption, the latter condition must hold.  In fact, we can assume $\lambda_1(L_g^1) < 0$, since otherwise we could take $t_0 = 1$.

Clearly, $\lambda_1(L_g^t)$ depends continuously on the parameter $t$.  Since $\Phi_g^0 = R_g$ and the Yamabe invariant of $(X^4,[g])$ is positive, we know that $\lambda_1(L_g^0) > 0$. By the intermediate value theorem, it follows there is $t_0 \in (0,1]$ with $\lambda_1(L_g^{t_0}) = 0$.  Also, integrating (\ref{Rprop}) and using the Cauchy-Schwarz inequality it is easy to see that $t_0$ satisfies (\ref{tb}).  \end{proof}

\subsection{The Proof of Theorem \ref{t:mainthm}}

In this subection use Theorem \ref{NestProp} to give the proof of Theorem \ref{t:mainthm}.  As remarked above, since the index and nullity are conformal invariants we are free to make a conformal modification of the base metric and we choose the conformal gauge guaranteed by Proposition \ref{CGF}.  To begin we obtain an algebraic estimate for the Jacobi operator.  Specifically, let $Z$ now denote the trace-free Ricci tensor, i.e.
\begin{align*}
\cZ := \Ric - \tfrac{1}{4} R g.
\end{align*}
We express $\mathcal{J}^{\N}$ as
\begin{gather} \label{J3}
\begin{split}
\mathcal{J}^{\N} =&\ -\Delta  + \tfrac{1}{4} R  + \cZ + 2 [ F_{\N}, \cdot ]\\
=&\ - \Delta + \tfrac{1}{6}R + \sqg{ \tfrac{1}{12}R + \tfrac{\sqrt{3}  }{12} \gamma_1  t_0  [ F, \cdot ] }_{\mathcal{A}} + \sqg{ \cZ + \prs{ 2 - \tfrac{\sqrt{3}}{12} \gamma_1 t_0} [F,\cdot ] }_{\mathcal{B}},
\end{split}
\end{gather}
and proceed to estimate the zeroth-order operators $\mathcal{A}$ and $\mathcal{B}$ labeled above.

\begin{lemma} \label{Alow} As a bilinear form, $\mathcal A \geq 0$.
\end{lemma}

\begin{proof}  If we take $\cZ=0$ and $\Phi = F_{\N}$ in Proposition \ref{Bsharp}, then
\begin{align*}
|F(A,A)| &= | \langle [ F, A], A \rangle| \\
&\leq \tfrac{\sqrt{3}}{2} \sqrt{ 2 \gamma_0^2  |F|^2 }|A|^2. \\
&= \tfrac{\sqrt{6}}{2} \gamma_0 |F| |A|^2.
\end{align*}
Since $\gamma_0 \leq \sqrt{2}$, it follows that
\begin{align*}
| \langle [ F, A], A \rangle| \leq \sqrt{3} |F| |A|^2.
\end{align*}
Therefore,
\begin{align*}
\mathcal{A}(A,A) &=  \tfrac{1}{12} R \brs{A}^2  + \tfrac{\sqrt{3}}{12} \gamma_1 t_0\ip{ \brk{ F, A }, A }  \\
&\geq \tfrac{1}{12} R |A|^2 - \tfrac{\sqrt{3}  }{12} \gamma_1 t_0 \prs{ \sqrt{3} |F|} |A|^2  \\
&= \tfrac{1}{12} \prs{ R - 3 \gamma_1  t_0 \brs{F} } \brs{A}^2.
\end{align*}
Using the formula for the scalar curvature in (\ref{Rprop}), we conclude
\begin{align*}
\mathcal{A}(A,A) &\geq \tfrac{1}{12} \prs{ R - 3 \gamma_1  t_0 \brs{F}  \brs{A}^2} \\
&= \tfrac{\sqrt{6}}{6} t_0 \brs{W} \brs{A}^2 \\
&\geq 0.
\end{align*}
\end{proof}

\begin{lemma} \label{Blow} Let
\begin{align} \label{alphadef}
\alpha = 2 - \tfrac{\sqrt{3}}{12} \gamma_1 t_0 > 0.
\end{align}
Then
\begin{align} \label{Bgeq}
\mathcal{B} \prs{A,A} \geq - \brk{\tfrac{3}{4} \brs{\cZ}^2 + 3 \alpha^2 \brs{F}^2}^{1/2} \brs{A}^2.
\end{align}
\end{lemma}

\begin{proof} Note that $\mathcal{B} = \cZ + \alpha [F, \cdot]$.  If we take $\Phi = \alpha F$ in Proposition \ref{Bsharp} and use the fact that $\gamma_0 \leq \sqrt{2}$, then
\begin{align*}
\mathcal{B}(A,A) &\geq -\tfrac{\sqrt{3}}{2} \brk{\brs{\cZ}^2 + 2 \gamma_0^2 \alpha^2 \brs{F}^2}^{1/2} \brs{A}^2 \\
&\geq \brk{\tfrac{3}{4} \brs{\cZ}^2 + 3 \alpha^2 \brs{F}^2}^{1/2}\brs{A}^2,
\end{align*}
as claimed.
\end{proof}

In view of (\ref{J3}) and Lemmas \ref{Alow} and \ref{Blow}, we have
\begin{align*} \begin{split}
\langle \mathcal{J}^{\nabla} A, A \rangle_{L^2} &\geq \langle \big(  - \Delta + \tfrac{1}{6}R  - \brk{\tfrac{3}{4} \brs{\cZ}^2 + 3 \alpha^2 \brs{F}^2}^{1/2} \big) A, A \rangle_{L^2} \\
&= \langle \prs{ - \Delta + \tfrac{1}{6}R - V }A, A \rangle_{L^2},
\end{split}
\end{align*}
where
\begin{align} \label{Vdef}
V = \brk{\tfrac{3}{4} \brs{\cZ}^2 + 3 \alpha^2 \brs{F}^2}^{1/2}.
\end{align}
We therefore define
\begin{align} \label{Kdef}
{\mathcal S} =  - \Delta + \tfrac{1}{6}R - V.
\end{align}

To estimate the index and nullity of $\mathcal J^{\N}$ it suffices to obtain the estimate for $\mathcal S$.  Applying Theorem \ref{NestProp} to the operator $\mathcal S$ on the bundle $\Lambda^{1} (\mathfrak g_E)$, which has rank $4d$, where $d = \dim(\mathfrak{g}_E)$, we obtain
\begin{align} \label{Nesty}  \begin{split}
N_0(\mathcal{S}) &\leq \dfrac{ 144 e^2 d   } {\cY(X^4, [g])^2 }   \int_{X} V^2 \, \dV \\
&\leq \dfrac{ 144 e^2 d   } {\cY(X^4, [g])^2 } \Bigg\{ \tfrac{3}{4}  \int_{X} |\cZ|^2 \,\dV + 3 \alpha^2  \int_{X} |F|^2 \, \dV   \Bigg\}.
\end{split}
\end{align}

By the Chern--Gauss--Bonnet formula
\begin{align}  \label{CGB}
\tfrac{3}{4} \int_X \brs{\cZ}^2 \,\dV = - 12 \pi^2 \Eul\prs{X^4} + \tfrac{3}{2} \int_X | W |^2 \,\dV + \tfrac{1}{16} \int_X R^2 \,\dV.
\end{align}
Using the conformal gauge fixing of Proposition \ref{CGF}, we can estimate the scalar curvature term above as
\begin{align*} \begin{split}
\tfrac{1}{16} \int_X R^2 \, \dV &= \tfrac{t_0^2}{16} \int_X \prs{ 2 \sqrt{6}|W| + 3 \gamma_1  |F| }^2 \, \dV \\
&= \tfrac{3}{2} t_0^2 \int_X |W|^2 \, \dV + \tfrac{3\sqrt{6}}{4} \gamma_1 t_0^2 \int_X |W||F| \, \dV + \tfrac{9}{16} \gamma_1^2 t_0^2 \int_X |F|^2 \, \dV.
\end{split}
\end{align*}
Substituting this into (\ref{CGB}) gives
\begin{align*} \begin{split}
\tfrac{3}{4} \int_X \brs{\cZ}^2 \,\dV &= - 12 \pi^2  \Eul \prs{X^4}  + \tfrac{3}{2}\prs{1 +  t_0^2} \int_X | W |^2 \,\dV \\
& \ \ \ \ + \tfrac{3\sqrt{6}}{4} \gamma_1 t_0^2 \int_X |W||F| \, \dV + \tfrac{9}{16} \gamma_1^2 t_0^2 \int_X |F|^2 \, \dV.
\end{split}
\end{align*}
We now substitute this into (\ref{Nesty}) to get
\begin{align} \label{Nesty2}  \begin{split}
N_0(\mathcal{\mathcal S}) &\leq \dfrac{ 144 e^2 d   } {\cY(X^4, [g])^2 } \Big\{ - 12 \pi^2 \Eul \prs{X^4}  +  \tfrac{3}{2}\prs{ 1 +  t_0^2} \int_X | W^{+} |^2 \,\dV \\
& \ \ \ \ + \tfrac{3\sqrt{6}}{4} \gamma_1 t_0^2  \int_X |W||F| \, \dV + \prs{ 3 \alpha^2 + \tfrac{9}{16} \gamma_1^2 t_0^2} \int_{X} |F|^2 \, \dV \Big\}.
\end{split}
\end{align}
We estimate the coefficients of each of terms above as follows:  For the first coefficient, since $t_0 \leq 1$ we have
\begin{align*}
\tfrac{3}{2}\prs{ 1 + t_0^2} \leq 3.
\end{align*}
Since $0 \leq t_0 \leq 1$ and by (\ref{g1}) $\gamma_1 \leq \tfrac{4\sqrt{3}}{3}$, we can bound the second coefficient by
\begin{align} \label{2ndc} \begin{split}
\tfrac{3\sqrt{6}}{4} \gamma_1 t_0^2 &\leq \tfrac{3\sqrt{6}}{4} \gamma_1 \\
&\leq 3 \sqrt{2}.
\end{split}
\end{align}
For the third coefficient we use the formula for $\alpha$ in (\ref{alphadef}) to write
\begin{align} \label{3rdc} \begin{split}
\prs{ 3 \alpha^2 + \tfrac{9}{16} \gamma_1^2 t_0^2} = \tfrac{5}{8} (\gamma_1 t_0)^2 - \sqrt{3} (\gamma_1 t_0) + 12.
\end{split}
\end{align}
Now $\gamma_1 t_0 \leq \frac{4\sqrt{3}}{3}$, and the quadratic polynomial $q(x) = \tfrac{5}{8} x^2 - \sqrt{3} x + 12$ attains its maximum at $x = 0$ on the interval $\brk{0, \frac{4 \sqrt{3}}{3}}$.  Consequently,
\begin{align*}
\prs{ 3\alpha^2 + \tfrac{9}{16} \gamma_1^2 t_0^2} \leq 12.
\end{align*}
With these estimates on the coefficients, we can rewrite (\ref{Nesty2}) as
\begin{align*}
N_0(\mathcal{\mathcal S}) \leq \dfrac{ 144 e^2 d   } {\cY(X^4, [g])^2 } \Big\{ - 12 \pi^2 \Eul\prs{X^4} + 3 \int_{X} | W |^2 \,\dV + 3 \sqrt{2} \int_X |W||F| \, \dV + 12  \int_{X} |F|^2 \, \dV \Big\},
\end{align*}
finishing the proof. \qed

\subsection{Linear growth rate in four dimensions}

Theorem \ref{t:mainthm} exhibits that the index can grow at worst linearly in the Yang-Mills energy of the connection.  In this section we show that this growth rate is sharp through an explicit family of examples.  Various authors \cite{SSU, BM, SS, Bor} have shown the existence of families of noninstanton Yang--Mills connection for a given $\SU(2)$ bundle over $\mathbb{S}^4$ provided that the charge $\gk$ satisfies $\kappa (E) \neq \pm 1$.  We will use the work of Sadun--Segert \cite{SS}, who constructed non-instanton Yang-Mills connections on the so-called `quadrupole bundles.'  The proposition below analyzes this construction in conjunction with an index estimate of Taubes (\cite{Taubes} Theorem 1.1) to exhibit the required index growth.
\begin{prop} \label{p:lineargrowth} Given $l = 4k - 1 > 1$, let $\N^l$ denote the Sadun--Segert connection on the quadrupole bundle $P_{(l,3)} \to \mathbb{S}^4$.  There exists a constant $\gd > 0$ so that
\begin{align*}
\imath\prs{\N^l} \geq \gd \brs{\brs{F_{\N^l}}}_{L^2}^2
\end{align*}
\begin{proof} We assume familiarity with the results and notation of \cite{SS}.  The quadrupole bundles are defined by different lifts of the unique irreducible representation of $\SU(2)$ on $\mathbb R^5$, and are classified by a pair of odd positive integers $(n_+, n_-)$, with the bundle denotes $P_{(n_+, n_-)}$.  The construction of \cite{SS} further restricts to the case $n_{\pm} \neq 1$.  We will choose $n_+ = l = 4k - 1 > 1$, $n_- = 3$, and let $\N^l$ denote the Sadun--Segert connection on $P_{(l,3)}$.  As computed in \cite{SS1, ASSS} one has
\begin{align} \label{f:lineargrowth10}
\gk(P_{(n_+, n_-)}) = \tfrac{1}{8} (n_+^2 - n_-^2) = \tfrac{1}{8} \left(l^2 - 9 \right).
\end{align}
Furthermore, as the connection $\N^l$ is not self-dual, \cite{Taubes} Theorem 1.1 yields
\begin{align} \label{f:lineargrowth20}
\imath(\N^l) \geq 2 \left( \brs{\gk(P_{(l, 3)})} + 1 \right)
\end{align}
We claim that there exists a constant $C > 0$ so that $\N^l$ satisfies
\begin{align} \label{f:lineargrowth15}
\brs{\brs{F_{\N^l}}}_{L^2}^2 \leq C l^2.
\end{align}
Assuming this for the moment, putting together (\ref{f:lineargrowth10}) - (\ref{f:lineargrowth15}) yields
\begin{align*}
\imath(\N^l) \geq&\ 2 \left( \brs{\gk(P_{(l, 3)})} + 1 \right)\\
=&\ 2 \left( \tfrac{1}{8} \left(l^2 - 9 \right) + 1 \right)\\
\geq&\ \tfrac{1}{4} l^2\\
\geq&\ \tfrac{1}{4C} \brs{\brs{F_{\N^l}}}_{L^2}^2,
\end{align*}
as required.

We now prove line (\ref{f:lineargrowth15}).  Connections with quadrupole symmetry on these bundles are described in terms of a triple of functions $a_i : (0,\frac{\pi}{3}) \to \mathbb R$, $i = 1,2,3$.  The bundle on which the connection is defined is determined by the boundary data.  In particular, as per (\cite{SS} Definition 2.5, Lemma 2.6), we require that $a = (a_1,a_2,a_3)$ satisfies
\begin{align} \label{f:lineargrowth30}
\lim_{\theta \to 0} a\prs{\theta} = \prs{0,0,l}, \qquad \lim_{\theta \to \frac{\pi}{3}} a(\theta) = \prs{0,3,0},
\end{align}
and moreover each $a_i$ extends to $(-\ge, \frac{\pi}{3} + \ge)$ such that for all $\theta \in (-\ge,\ge)$,
\begin{gather} \label{f:lineargrowth40}
\begin{split}
a_1\prs{\theta} = a_2\prs{-\theta},& \qquad a_3 \prs{\theta} = a_3\prs{-\theta}\\
a_1\prs{\tfrac{\pi}{3} + \theta} = a_3(\tfrac{\pi}{3} - \theta),& \qquad a_2\prs{\tfrac{\pi}{3} + \theta} = a_2\prs{\tfrac{\pi}{3} - \theta}
\end{split}
\end{gather}
We can construct a test connection which satisfies these conditions as follows.  First set $a_1 \equiv 0$.  Fix some small $\gd > 0$ and define $a_2$ via
\begin{align*}
a_2(\theta) \equiv&\ 0 \qquad \mbox{for }  \theta \in (-\gd,\gd)\\
a_2(\theta) \equiv&\ 3 \qquad \mbox{for }  \theta \in \left(\tfrac{\pi}{3} - \gd, \tfrac{\pi}{3} + \gd \right)\\
0 \leq a_2(\theta) \leq&\ 3 \qquad \mbox{for } \theta \in [0,\tfrac{\pi}{3}]\\
0 \leq a_2'(\theta) \leq&\ 5 \qquad \mbox{for } \theta \in [0,\tfrac{\pi}{3}].
\end{align*}
and we define $a_3$ via
\begin{align*}
a_3(\theta) \equiv&\ 3 \qquad \mbox{for }  \theta \in (-\gd,\gd)\\
a_3(\theta) \equiv&\ 0 \qquad \mbox{for }  \theta \in \left(\tfrac{\pi}{3} - \gd, \tfrac{\pi}{3} + \gd \right)\\
0 \leq a_3(\theta) \leq&\ 3 \qquad \mbox{for } \theta \in [0,\tfrac{\pi}{3}]\\
0 \geq a_3'(\theta) \geq &\ -5 \qquad \mbox{for } \theta \in [0,\tfrac{\pi}{3}].
\end{align*}
One easily checks that this satisfies conditions (\ref{f:lineargrowth30}) and (\ref{f:lineargrowth40}) for $l = 3$.  Furthermore, if we set, for $l > 0$,
\begin{align*}
a_{l} := \left(a_1, a_2, \tfrac{l}{3} a_3 \right)
\end{align*}
then $a_{l}$ satisfies the conditions of (\ref{f:lineargrowth30}) and (\ref{f:lineargrowth40}) for the $(l,3)$ bundle,
and furthermore satisfies
\begin{align*}
0 \leq a_3(\theta) \leq l, \qquad 0 \geq a_3'(\theta) \geq -5 l.
\end{align*}

In (\cite{SS} Proposition 2.7) the Yang-Mills energy of these connections is computed, and takes the form
\begin{gather} \label{f:lineargrowth50}
\begin{split}
\brs{\brs{F_{\N(a)}}}_{L^2}^2 =&\ \pi^2 \int_0^{\tfrac{\pi}{3}} \left[ (a_1')^2 G_1 + (a_1 + a_2 a_3)^2/G_1 + (a'_2)^2 G_2 + (a_2 + a_1a_3)/G_2 \right.\\
&\ \left. \qquad \qquad + (a_3')^2 G_3 + (a_3 + a_1 a_2)^2/G_3 \right] \cd\theta,
\end{split}
\end{gather}
where
\begin{align*}
G_1 =&\ \tfrac{f_2 f_3}{f_1}, \qquad G_2 = \tfrac{f_3 f_1}{f_2}, \qquad G_3 = \tfrac{f_1 f_2}{f_3}\\
f_1\prs{\theta} =&\ 2 \sin \prs{\tfrac{\pi}{3} + \theta}, \qquad f_2 \prs{\theta} = 2 \sin \prs{\tfrac{\pi}{3} - \theta}, \qquad f_3\prs{\theta} = 2 \sin \prs{\theta}.
\end{align*}
Note that some terms in the energy formula involve factors of the $G_i$ which can blowup at one endpoint or the other, but the boundary conditions for $a$ ensure that these are finite integrals.  In particular, for our initial choice of $a = a_3$, we obtain some value for the Yang-Mills energy, call it $C$.  We furthermore observe that every term in (\ref{f:lineargrowth50}) is at worst quadratic in $a_3$ and $a_3'$, which both grow linearly with $l$, and hence it follows that there is a different constant $C$ such that
\begin{align*}
\brs{\brs{F_{\N(a_{l})}}}_{L^2}^2 \leq C l^2.
\end{align*}
As the Sadun--Segert connection is constructed by energy minimization within this symmetry class (\cite{SS} Proposition 3.4, Theorem 3.10), its energy must lie below that of this test connection, finishing the proof of (\ref{f:lineargrowth15}).
\end{proof}
\end{prop}

\section{The Index of a positive Einstein Metric} \label{s:Einstein}

Let $X^4$ be a smooth, closed, four-dimensional manifold.  Furthermore suppose $g$ is a critical point for the normalized total scalar curvature functional given in \eqref{EH}:
\begin{align*}
\mathscr{S}[g] = \Vol(g)^{-1/2} \int_{X^4} \R_g \,\dV_g,
\end{align*}
where $R_g$ is the scalar curvature of $g$.  It follows that $g$ is an Einstein metric, whose Ricci tensor is given by
\begin{align*}
\Ric(g) = \tfrac{1}{4} R g
\end{align*}
(see \cite{Besse}, Chapter 4C).

To study the second variation of $\mathcal{S}$ at $g$, one uses the splitting of the space of sections of the bundle of symmetric two-tensors (see \cite{SchoenCIME} for details).  The stability operator, corresponding to transverse-traceless variations of $g$, is given by
\begin{align} \label{StabOp} \begin{split}
\prs{\mathcal{L}(h) }_{ij} &= \Delta h_{ij} + 2 R_{ik j \ell} h_{k \ell} \\
&= \Delta h_{ij} + 2 W_{ik j \ell} h_{k \ell} - \tfrac{1}{6}R h_{ij}.
\end{split}
\end{align}
This defines an index form
\begin{align} \label{DerLL} \begin{split}
I(h,h) &= \int_X \langle h, \mathcal{L} ( h) \rangle \,\dV \\
&= \int_X \big[ - |\nabla h|^2 + 2 W(h,h) - \tfrac{1}{6} R |h|^2 \big] \,\dV,
\end{split}
\end{align}
where
\begin{align*}
W(h,h) = W_{ik j \ell}  h_{k \ell} h_{ij}.
\end{align*}
The index $\imath(g)$ of an Einstein metric is the number of positive eigenvalues of $\mathcal{L}$ (equivalently, the number of negative eigenvalues of $-\mathcal{L}$).  The nullity $\nu(g)$ of an Einstein metric is the dimension of the kernel of $\mathcal{L}$, i.e., the dimension of the space of infinitesimal Einstein deformations (see Chapter 12 of \cite{Besse}) .  With this background we can give the proof of Theorem \ref{t:mainthm2}.

\begin{proof}[Proof of Theorem \ref{t:mainthm2}] Note that $\mathcal{L} : S_0^2(T^{*}X^4) \rightarrow S_0^2(T^{*}X^4)$, where $S_0^2(T^{*}X^4)$ is the bundle of trace-free symmetric two-tensors.  It follows from (\cite{Huisken}, Lemma 3.4), that\footnote{Note that in \cite{Huisken}, the norm of Weyl is the one induced by the metric on covariant $4$-tensors, while we are using the norm of Weyl viewed as a section of $\mbox{End}(\Lambda^2)$.}
\begin{align*}
- W(h,h) \geq -\tfrac{2}{\sqrt{3}} |W||h|^2.
\end{align*}
Therefore,
\begin{align*} \begin{split}
\int_X \langle h, - \mathcal{L} (h) \rangle \,\dV &\geq \int_X \big[ |\nabla h|^2 - \tfrac{4}{\sqrt{3}} |W||h|^2 + \tfrac{1}{6} R |h|^2 \big] \,\dV \\
&= \int_X \langle h, \big( -\Delta + \tfrac{1}{6}R - V \big)h \rangle \,\dV,
\end{split}
\end{align*}
where
\begin{align*}
V = \tfrac{4}{\sqrt{3}} |W|.
\end{align*}
Since $\dim(S_0^2(T^{*}X^4)) = 9$, applying Theorem \ref{NestProp} to the operator $\mathcal{N} =  -\Delta + \frac{1}{6}R - V$ gives
\begin{align} \label{INE}
\imath(g) + \nu(g) \leq 1728 e^2 \dfrac{ \int_X |W|^2 \,\dV } {\cY(X^4, [g])^2 }.
\end{align}
Since $g$ is Einstein,
\begin{align} \label{YS}
\cY(X^4, [g]) = \mathscr{S}[g].
\end{align}
Also, by the Chern--Gauss--Bonnet formula,
\begin{align*}
8 \pi^2 \Eul(X^4) &= \int_X \big( |W|^2 +  \tfrac{1}{24} R^2 \big) \,\dV \\
&= \int_X |W|^2 \,\dV + \tfrac{1}{24} \mathscr{S}[g]^2.
\end{align*}
Substituting this into (\ref{INE}), using (\ref{YS}), and rearranging the inequality gives
\begin{align*}
\mathscr{S}[g] \leq 24 \pi \sqrt{ \dfrac{ \Eul(X^4)}{ 3  +  \delta \brk{\imath(g) + \nu(g)} }},
\end{align*}
where $\delta = (24 e^2)^{-1}$, as required.
\end{proof}

\section{The proof of Theorem \ref{t:mainthm3}}

\begin{proof}[Proof of Theorem \ref{t:mainthm3}]

Let $(X^4,g)$ be an oriented four-manifold with positive scalar curvature.
To obtain the estimate for the first Betti number we only need to make minor changes to the index estimate for Yang-Mills connections, since the Jacobi operator in the case of the trivial bundle is the Hodge Laplacian acting on $\Lambda^1$.  The only difference is the choice of conformal representative: in the trivial case, we use a Yamabe metric in the conformal class of $g$ instead of the metric specified in Proposition \ref{CGF}.

Let $\mathcal{H}_1 : \Lambda^1 \rightarrow \Lambda^1$ denote the Hodge Laplacian.  Then by the Hodge-de Rham theorem, $H^1(X^4,\mathbb{R}) = \ker \mathcal{H}_1$, and $\dim \ker \mathcal{H}_1 = b_1(X^4)$.  Let $\omega \in H^1(X^4,\mathbb{R})$ be a harmonic one-form; by the classical Bochner formula,
\begin{align*}
\langle -\mathcal{H}_1 \omega, \omega \rangle_{L^2} &= \int_{X} \prs{ |\nabla \omega|^2 + \Ric(\omega,\omega) } \, \dV \\
&= \int_{X} \prs{ |\nabla \omega|^2 + \tfrac{1}{4}R |\omega|^2 + \cZ(\omega,\omega)  } \, \dV \\
&\geq \int_{X} \prs{ |\nabla \omega|^2 + \tfrac{1}{6}R |\omega|^2 + \cZ(\omega,\omega)  } \, \dV.
\end{align*}
Since $\cZ$ is trace-free,
\begin{align*}
\cZ(\omega,\omega) \geq -\tfrac{\sqrt{3}}{2}|\omega|^2.
\end{align*}
Therefore,
\begin{align*}
\langle -\mathcal{H}_1 \omega, \omega \rangle_{L^2} &\geq \int_{X} \prs{ |\nabla \omega|^2 + \tfrac{1}{12} R |\omega|^2 - \tfrac{\sqrt{3}}{2}|\cZ| |\omega|^2  } \, \dV \\
&= \big\langle \prs{ -\Delta + \tfrac{1}{6}R - V }\omega, \omega \big\rangle_{L^2},
\end{align*}
where
\begin{align*}
V = \tfrac{\sqrt{3}}{2}|\cZ|.
\end{align*}

Applying Theorem \ref{NestProp} to the operator $-\Delta + \frac{1}{6}R - V$ with $\mathcal{E} = \Lambda_1$, we get
\begin{align} \label{BP1}
b_1(X^4) &\leq \frac{108 e^2}{\cY(X^4,[g])^2} \int_{X} |\cZ|^2 \, \dV.
\end{align}
Recall
\begin{align*}
\rho_1(X^4,[g]) = \dfrac{ 4 \int \sigma_2(A_g) \, \dV}{\cY(X^4,[g])^2} = \dfrac{ \int_{X} \prs{ -\tfrac{1}{2}|\cZ|^2 + \tfrac{1}{24}R^2 }\, \dV}{\cY(X^4,[g])^2}.
\end{align*}
Since $g$ is a Yamabe metric,
\begin{align*}
\int_{X} R^2 \,\dV = \cY(X^4,[g])^2.
\end{align*}
Consequently,
\begin{align*}
\int_{X} |\cZ|^2 \, \dV &= - 2\rho_1(X^4,[g]) \cY(X^4,[g])^2 + \tfrac{1}{12} \int_{X} R^2 \, \dV \\
&= \tfrac{1}{12} \prs{ 1 - 24 \rho_1(X^4,[g]) } \cY(X^4,[g])^2.
\end{align*}
Substituting this into (\ref{BP1}) gives (\ref{b1}).

To estimate $b^{+}(X^4)$, let $\mathcal{H}_2 : H^2(X^4) \rightarrow H^2(X^4)$ denote the Hodge Laplacian.  Then $b^{+}(X^4) = \dim \ker \mathcal{H}_2^{+}$, where $\mathcal{H}^{+}_2$ is the restriction of $\mathcal{H}_2$ to $\Lambda_{+}^2$, the bundle of self-dual two-forms.  The space of self-dual harmonic two-forms is conformally invariant since the Hodge $\star$ operator is.  Therefore, in estimating $b^+(X^4)$ we are free to choose a conformal metric.  If we take the bundle $E$ to be the trivial bundle in Proposition \ref{CGF}, then there is a conformal metric $\hg \in [g]$ and a $t_0 \in (0,1]$ such that
\begin{align} \label{Rprop2}
R_{\hg} = 2\sqrt{6} t_0 |W^{+}_{\hg}|.
\end{align}
From now on we assume $g = \hg$.

The operator $\mathcal{H}^+_2$ satisfies the Weitzenbock formula
\begin{align*}
\mathcal{H}_2^{+} = \Delta + 2 W^{+} - \tfrac{1}{3}R,
\end{align*}
where $\Delta$ is the rough Laplacian.  Since $W^{+}: \Lambda^2_{+} \rightarrow \Lambda^2_{+}$ is trace-free and $\dim \Lambda^2_{+} = 3$, we have the sharp inequality
\begin{align*}
| W^{+}(\omega,\omega)| \leq \tfrac{2}{\sqrt{6}}|W^{+}| |\omega|^2.
\end{align*}
Therefore,
\begin{align*}
\langle -\mathcal{H}_2^{+} \omega, \omega \rangle_{L^2} &= \int_{X} \prs{ |\nabla \omega|^2 - 2 W^{+}(\omega,\omega) + \tfrac{1}{3}R|\omega|^2 } \, \dV \\
&\geq \int_{X} \prs{ |\nabla \omega|^2 - \tfrac{4}{\sqrt{6}} |W^{+}||\omega|^2 + \tfrac{1}{3}R|\omega|^2 } \, \dV \\
&= \int_{X} \prs{ |\nabla \omega|^2 + \tfrac{1}{6} R |\omega|^2 + \prs{ \tfrac{1}{6}R - \tfrac{4}{\sqrt{6}} |W^{+}| }|\omega|^2  } \, \dV.
\end{align*}
Using (\ref{Rprop2}),
\begin{align*}
\langle -\mathcal{H}_2^{+} \omega, \omega \rangle_{L^2} &\geq \int_{X} \prs{ |\nabla \omega|^2 + \tfrac{1}{6} R |\omega|^2 - \tfrac{\sqrt{6}}{3}(2 - t_0)  |W^{+}| |\omega|^2  } \, \dV \\
&\geq \big\langle \prs{ -\Delta + \tfrac{1}{6}R - V }\omega, \omega \big\rangle_{L^2},
\end{align*}
where
\begin{align*}
V = \tfrac{\sqrt{6}}{3}(2 - t_0)  |W^{+}|.
\end{align*}

Applying Theorem \ref{NestProp} to the operator $-\Delta + \frac{1}{6}R - V$ with $\mathcal{E} = \Lambda_2^{+}$, we get
\begin{align} \label{BP} \begin{split}
b+(X^4) &\leq \frac{72 e^2}{\cY(X^4,[g])^2} (2-t_0)^2 \| W^{+}\|_{L^2}^2 \\
&= 3 e^2 (2 - t_0)^2 \rho_{+}(X^4,[g]),
\end{split}
\end{align}
where $\rho_{+}$ is given by (\ref{rhodef}).  By (\ref{tb}) of Proposition \ref{CGF},
\begin{align*}
\rho_{+}^{-1/2} \leq t_0 \leq 1,
\end{align*}
hence
\begin{align*}
\prs{ 2 - t_0}^2 \rho_{+} \leq ( 2 - \rho_{+}^{-1/2} )^2 \rho_{+} \leq ( 2 \rho_{+}^{1/2} - 1 )^2.
\end{align*}
Substituting this into (\ref{BP}) gives (\ref{bplus}).

\end{proof}


\end{document}